\documentclass[11pt,a4paper,subeqn,reqno]{article}

\usepackage{amsmath,amsthm}
\usepackage{amssymb}
\usepackage[mathscr]{eucal}
\usepackage[all]{xy}
\usepackage{mathrsfs}
\usepackage{hyperref}
\usepackage{appendix}
\usepackage{authblk}
\usepackage{graphicx}
 \usepackage[misc]{ifsym}

\usepackage{url}
\usepackage{hyperref}

\allowdisplaybreaks 

\usepackage[text={165mm,250mm},centering]{geometry}

\newtheorem{theorem}{Theorem}[section]
\newtheorem{lemma}[theorem]{Lemma}

\newtheorem{corollary}[theorem]{Corollary}
\newtheorem{proposition}[theorem]{Proposition}

\theoremstyle{definition}
\newtheorem{example}[theorem]{Example}
\newtheorem{definition}[theorem]{Definition}
\newtheorem{definition-lemma}[theorem]{Definition-Lemma}
\newtheorem{definition-theorem}[theorem]{Definition-Theorem}

\newtheorem{remark}[theorem]{Remark}

\newtheorem{proof*}{Proof}

\begin{document}

\title{Non-Embedding Theorems of Nilpotent Lie Groups and Sub-Riemannian Manifolds}

\author[1,3]{Yonghong Huang\thanks{(\Letter) Email: yonghonghuangmath@gmail.com. Partially supported by NSFC (No. 11771303).}}

\author[1,2]{Shanzhong Sun\thanks{Email: sunsz@cnu.edu.cn. Partially supported by NSFC (No. 11771303) and Beijing Advanced Innovation Center for Imaging Theory and Technology, Capital Normal University.} }

\renewcommand\Affilfont{\small}

\affil[1]{Department of Mathematics, Capital Normal University,
Beijing 100048 P. R. China}
\affil[2]{Academy for Multidisciplinary Studies, Capital Normal  University, Beijing 100048 P. R.
China}
\affil[3]{Qiannan Preschool Education College for
Nationalities, Guiding Guizhou 558000 P. R.
China}

\date{}

\maketitle
\renewcommand{\abstractname}{Abstract}

\begin{abstract}
In this paper, we prove that there do not exist quasi-isometric
embeddings of connected non-abelian nilpotent Lie groups equipped
with left invariant Riemannian metrics into a metric measure space
satisfying the curvature-dimension condition $RCD(0,N)$ with
$N\in\mathbb{R}$ and $N>1$. In fact, we can prove that
a sub-Riemannian manifold whose generic degree of
nonholonomy is not smaller than $2$ can not be bi-Lipschitzly
embedded in any Banach space with the Radon-Nikodym property. 
 We also get that every regular
sub-Riemannian manifold do not satisfy the curvature-dimension
condition $CD(K,N)$, where $K,N\in\mathbf{R}$ and $N>1$. Along the way to the proofs, we show that the minimal weak upper gradient (Def. \ref{defwg}) and the horizontal gradient (Def. \ref{defhg}) coincide on the Carnot-Carath\'eodory spaces which may have independent interests (Thm. \ref{thmih}).
\end{abstract}

\section{Introduction}

In \cite{sp}, Pauls proved that there do not exist quasi-isometric embeddings of
any connected, simply connected nonabelian nilpotent Lie groups
equipped with left invariant Riemannian metric into either a $CAT_0$
metric space or an Alexandrov metric space with nonnegative
curvature bounded. Inspired by Pauls' work, we address ourselves the following problem whether such connected nilpotent
Lie groups equipped with left invariant Riemannian metrics can be
quasi-isometricaly embedded into metric measure spaces.
A metric measure space is a triple $(M,d,m)$, where $(M,d)$ is a complete and separable metric
space and $m$ is a locally finite (i.e., $m(B_{r}(x))<\infty$
for all $x\in M$ and for all sufficiently small $r>0$) nonnegative complete Borel measure on $M$
equipped with its Borel $\sigma$-algebra. To avoid pathologies, we
exclude the case $m(M)=0$. 
Let $(G,g)$ be a connected
nilpotent Lie group with a left invariant Riemannian metric and $d$
be the induced distance function on $G$. If $(X,d_X)$ is a complete
metric space, then ${f:G\rightarrow X }$ is called an
$(L,C)$-quasi-isometric embedding if, for all $x,y \in G$,
\[ \frac{1}{L}d(x,y)-C
 \le
d_{X}(f(x),f(y))\le
  {L}d(x,y)+C.\]
We get the following theorem
\begin{theorem}\label{thm1.1}
There does not exist a quasi-isometric embedding of any nonabelian nilpotent Lie group equipped with a left invariant Riemannian metric
into any metric measure space satisfying $RCD(0,N)$, with $K,N\in\mathbf{R}$ and $N>1$, where $(M,d)$ is a length space and $supp(m)=M$.
\end{theorem}
It is very challenging to define metric spaces with a lower bound on
the curvature. Alexandrov \cite{ad} introduced the notion of lower
sectional curvature bound for metric spaces in terms of comparison
properties for geodesic triangles. For Ricci curvature bound, an
amazing theory has recently been developed independently by Sturm
\cite{kts} and Lott and Villani \cite{jl} in terms of
curvature-dimension condition $CD(K,N)$, where $K$ is the lower
bound of Ricci curvature and $N$ is the upper bound of Hausdorff
dimension. Sub-Riemannian manifolds with their Carnot-Carath\'eotory
metrics and Riemannian volumes are natural metric measure spaces.
Through the method from metric geometry, we can prove the following result
\begin{theorem}\label{thm1.2}
Let $(M,\Delta,g)$ be a regular sub-Riemannian manifold. Then $M$ with the Carnot-Carath\'eodory
metric $d$ induced by $g$ and a Riemannian volume, as a metric measure space does not satisfy any $CD(K,N)$, where $K,N\in\mathbf{R}$
and $N>1$.
\end{theorem}
Our main observation to prove the theorem is that in this case, the curvature-dimension condition $CD(K,N)$ implies $RCD(K,N)$ (Def. \ref{defrcd}). In fact, we prove that for any sub-Riemannian manifold with Carnot-Carath\'eodory metric and fixed Riemannian volume, the horizontal gradient and the minimal weak gradient coincide. More precisely, we prove the following
\begin{theorem}
Let $(M,\Delta,g)$ be a regular sub-Riemannian manifold with measure $m$ induced by Riemannian volume and Carnot-Carath\'{e}odory distance $d$. Then the two spaces $W^{1,2}(M,d,m)$ and $W_{H}^{1,2}(M,m)$ (Def. \ref{defhg}) coincide. So $(M,d,m)$ is an infinitesimally Hilbert space (Def. \ref{defihs}).
\end{theorem}
This is our Theorem \ref{thmih}, which together with the following Theorem \ref{thm 1.3}  forms the key ingredient to the proof of Theorem \ref{thm1.2}.

Cheeger and Kleiner \cite{jc} proved that the Heisenberg group with
the Carnot-Carath\'eodory metric does not admit a bi-Lipschitz
embedding into any Banach space satisfying the Radon-Nikodym
property (see Def. \ref{defrnp}). By refining their methods and ideas, we prove that the same result holds in the case of Carnot groups.
Then, applying a blow up argument, we have the following theorem


\begin{theorem}\label{thm 1.3} Any sub-Riemannian manifold whose generic
degree of non-holonomy is not smaller than $2$ can not be
 bi-Lipschitzly embedded into a Banach space with
the Radon-Nikodym property.
\end{theorem}

The following corollary is a byproduct of the Theorem \ref{thm 1.3}. Its proof will be given in \S \ref{sect gh}.

\begin{corollary}\label{coroesrnercd}
Any complete regular sub-Riemannian manifold can not be bi-Lipschitzly embedded into a metric measure space $(M,d,m)$ satisfying $RCD(K,N)$, with $K,N\in\mathbf{R}$ and $N>1$, where $(M,d)$ is a length space and $supp(m)= M$.
\end{corollary}

By direct calculations, Juillet \cite{nj} showed that no
curvature-dimension bound $CD(K,N)$ holds in any Heisenberg group
$\mathbb{H}_n$ with its Carnot-Carath\'eodory distance and Lebesgue
measure. The Heisenberg group with Carnot-Carath\'eodory metric is a
basic example of sub-Riemannian manifold (see \S \ref{sect sc}). Our proof of Theorem \ref{thm1.2} is based on
a blow-up argument and by contradiction: if a sub-Riemannian manifold satisfies $CD(K,N)$ (see Def. \ref{defcd}), then it would be a $RCD(K,N)$ space (see Def. \ref{defrcd}) and thus by structure theory of $RCD(K,N)$ spaces the tangent cones would be Euclidean. On the other hand the tangent cones to sub-Riemannian manifolds are Carnot groups, and the contradiction is due to our Theorem \ref{thm 1.3}. The Gromov-Hausdorff convergence theory for $L$-biLipschitz maps between locally compact length spaces plays an essential role in this argument. Concerning Theorem \ref{thm 1.3}, for the sub-Riemannian manifold whose generic degree of non-holonomy is $1$, it may have a bi-Lipschitz embedding into some Banach space with the Radon-Nikodym property. In fact, in a preprint \cite{js}, Seo proved that the Grushin plane can be bi-Lipschitzly embedded in some Euclidean space.

The paper is organized as follows.  In \S \ref{sect sc},  we review the nilpotent Lie groups and sub-Riemannian geometry. We prove Theorem \ref{thm 1.3} in \S \ref{sect cg} for the special case that the sub-Riemannian manifold is a Carnot group (Theorem \ref{thmcn}) to which the general case will be reduced to in \S \ref{sect gh} by Gromov-Hausdorff convergence theory for $L$-biLipschitz maps between locally compact length spaces. The section \S \ref{sect cd} is a review of some basic properties of curvature-dimension condition for later
use.
Finally, in \S \ref{sect cs}, we recall the definitions of
Cheeger energy and Sobolev spaces on metric measure spaces and then
prove Theorem 1.3, Theorem \ref{thm1.2} and Theorem \ref{thm1.1}.

All the main results of the paper were announced in \cite{hs}. In fact, Theorem \ref{thm1.2} and Theorem \ref{thm 1.3} were derived and firstly announced at a seminar run by Professor Fuquan Fang in 2014. While we submitted \cite{hs} to the arXiv, we also noticed that Ambrosio and Stefani claimed a similar result (Proposition 3.6 in \cite{ag}) as our Theorem \ref{thm1.2} using different method of proof.

\section{Nilpotent Lie groups and Sub-Riemannian manifolds}\label{sect sc}

In this section, we recall some basic definitions of nilpotent Lie
groups and sub-Riemannian geometry. For a detailed discussions of these
topics, please refer to \cite{aa}.

\begin{definition}[Nilpotent Lie algebra]\label{defnla}
Let $(\mathfrak{g},[\cdot,\cdot])$ be a Lie algebra over
$\mathbb{R}$. We define the lower central sequence
$C^k(\mathfrak{g})$ as follows: for any natural number $k$,
\begin{equation}
\begin{split}
C^0(\mathfrak{g})&:= \mathfrak{g},\\
C^{k+1}(\mathfrak{g})&:=[\mathfrak{g},C^k(\mathfrak{g})].
\end{split}
\end{equation}
We say $\mathfrak{g}$ is a nilpotent Lie algebra, if
 $C^k(\mathfrak{g})=0$ for some $k$.
\end{definition}

Nilpotent Lie groups are defined via their Lie algebras.

\begin{definition}[Nilpotent Lie group]\label{defnlg}
Let $G$ be a connected Lie group with Lie algebra $\mathfrak{g}$.
Then $G$ is called nilpotent if its Lie algebra $\mathfrak{g}$ is
nilpotent.
\end{definition}

Carnot groups are a special class of nilpotent Lie groups
which play a central role in this paper.

\begin{definition}[Carnot group]\label{defcg}
A Carnot group is a simply connected
Lie group $G$ with a distinguished vector space $V_1$ such that the
Lie algebra $\mathfrak{g}$ of the group has the direct sum decomposition

\begin{equation}
\begin{split}
  \mathfrak{g} = \sum\limits_{k=1}^{m}V_k,\qquad  V_{k+1}=[V_1,V_k].
\end{split}
\end{equation}
The number $m$ is called the step of the group.
\end{definition}

For more motivations and backgrounds on the Carnot groups, please
refer to Pansu \cite{ppc,pp} and Mitchell (\cite{jm}).

\begin{remark}\label{remarkcgnab}
In the following, all Carnot groups that we will use are {\it nonabelian}.
\end{remark}

\begin{example}[Heisenberg group]\label{exmhg}
The Heisenberg group $\mathbb{H}$ is the group of  $3\times3$ upper
triangular matrices of the form
\[\begin{matrix}
\begin{bmatrix} 1 & x & z \\ 0 & 1 & y \\ 0 & 0 & 1
\end{bmatrix}
\end{matrix}\]
under the operation of matrix multiplication.
The Lie algebra $\mathfrak{h}$ of $\mathbb{H}$ is generated by
\[
\begin{matrix}
X=\begin{bmatrix} 0 & 1 & 0 \\ 0 & 0 & 0 \\ 0 & 0 & 0
\end{bmatrix}, \,\,\,
Y=\begin{bmatrix} 0 & 0 & 0 \\ 0 & 0 & 1 \\ 0 & 0 & 0
\end{bmatrix},\,\,\,
Z=\begin{bmatrix} 0 & 0 & 1 \\ 0 & 0 & 0 \\ 0 & 0 & 0
\end{bmatrix},\end{matrix}\] with the Lie bracket the usual commutator of matrices.

The Heisenberg group is a Carnot group. In fact, the generators of the Lie algebra satisfy the unique nontrivial relation $[X,Y]=Z$. Thus the Lie
algebra $\mathfrak{h}$ has a direct sum decomposition
\begin{equation*}
\mathfrak{h}= V_1\oplus V_2, \text{where $V_1=\textrm{span}\{X,Y\}$,
$V_2=\textrm{span}\{Z\}$}.
\end{equation*}

\end{example}

\begin{definition}\label{defsm}
A sub-Riemannian manifold is a triple $(M,\Delta,g)$ such that $M$
is a smooth manifold, $\Delta$ is a bracket-generating distribution
and $g$ is an inner product defined on $\Delta$. The metric g is
also called a sub-Riemannian metric of manifold $M$.
\end{definition}

Recall that a distribution $\Delta$ is said to be a bracket-generating
distribution if for any point $p \in M $, $ \exists ~~k= k(p)\in
\textbf{N} $ such that
\begin{eqnarray*}
\Delta^1&=&\Delta, \\
 \Delta^2 &=& \Delta^1+[\Delta,\Delta], \\
 \Delta^3 &=& \Delta^2+[\Delta,\Delta^2],\\
\Delta^4 &=& \Delta^3+[\Delta,\Delta^3],\\
 &\cdots& \\
 \Delta^k &=& \Delta^{k-1}+[\Delta,\Delta^{k-1}],\\
\Delta^k_p &=& T_pM.
\end{eqnarray*}

\begin{remark}\label{remarknhd}
The smallest integer $k=k(p)$ such that $\Delta_{p}^{k(p)}= T_{p}M$
is called the degree of non-holonomy  at $p$.
\end{remark}

\begin{definition}\label{defrp}
On a sub-Riemannian manifold $(M,\Delta,g)$, we say that a point
$p\in M$ is a regular point if $\dim\Delta^i$ remains constant in
some neighborhood of $p$ for $0\le i \le k(p)$. Otherwise we call
$p$ a singular point.
\end{definition}

\begin{remark}\label{remarkrpd}
The regular point is defined by open condition. So the set of regular points is open and dense in $M$. For more details, the reader is referred to \cite{ab} (the arguments in the last paragraph of p.$31$). If all points
are regular, we call the sub-Riemannian manifold is regular.
\end{remark}

Here are some well known examples of sub-Riemannian manifolds.

\begin{example}

By Example \ref{exmhg}, the Heisenberg group $\mathbb{H}$ with the
distribution $\Delta$ generated by $X$ and $Y$, and Euclidean metric given
by $\langle X,Y\rangle=0$, $\langle X,X\rangle=1$, $\langle
Y,Y\rangle=1$ is a sub-Riemannian manifold.
\end{example}

\begin{example}[Hopf fiberation]
The special unitary group $SU(2)$ is the Lie group consisting of $2\times2$ unitary
matrices of determinant $1$. Its Lie algebra $\mathfrak{su(2)}$ consists of skew
Hermitian matrices of trace zero. The distribution $\Delta$
generated by left invariant vector fields of the following two
elements in $su(2)$:
\[\begin{matrix}
v_1=\begin{bmatrix} 0 & 1/2\\ -1/2 & 0
\end{bmatrix}, \,\,\,\,&
v_2=\begin{bmatrix} 0 & i/2 \\ i/2 & 0
\end{bmatrix},
\end{matrix}\]
is a bracket-generating distribution. The standard sub-Riemannian
metric $g_0$ on $\Delta$ is given by $\langle
v_i,v_j\rangle=\delta_{ij}$, where $i,j=1,2$. We denote the
sub-Riemannian manifold by $(SU(2),\Delta,g_0)$. The one parameter
subgroup generated by the left invariant vector field
\[\begin{matrix}
v_3=\begin{bmatrix} i/2 & 0 \\ 0 & -i/2
\end{bmatrix}
\end{matrix}\]
gives a (right) $S^1$-action on $SU(2)$. The quotient of $SU(2)$ by
the $S^1$-action is the standard $2$-sphere $S^2$, which is the
Hopf fiberation.
\end{example}

\begin{example}[Special linear group]

The special linear group $SL(2, \textbf{R})$ consists of $2\times2$  matrices of
determinants $1$.
The distribution $\Delta$ is defined by
\[\begin{matrix}
X=\begin{bmatrix} 1/2 & 0 \\ 0 & -1/2
\end{bmatrix},&
Y=\begin{bmatrix} 0 & 1/2 \\ 1/2 & 0
\end{bmatrix}
\end{matrix}.\]
And the standard sub-Riemannian metric on $\Delta$ is given by $\langle
X,Y\rangle=0$, $\langle X,Y\rangle=1$, $\langle X,Y\rangle=1$.
\end{example}

\begin{example}[The Grushin plane]

The Grushin plane is the plane $\textbf{R}^2$ with the horizontal
distribution spanned by two vector fields
 \begin{displaymath}
 X_1 =\frac{\partial}{\partial x} ~~ \textrm{and} ~~ X_2 = x\frac{\partial}{\partial y}.
 \end{displaymath}
The vector fields $X_1$ and $X_2$ form an orthonormal basis for the
tangent space at each point in
$\textbf{R}^2\backslash\{(x,y):x=0\}$.
\end{example}

\begin{definition}\label{defhc}
An absolutely continuous curve ${\gamma:[a,b]\rightarrow M }$ on
${(M,\Delta,g)}$ is called horizontal if it is almost everywhere
tangent to the distribution $\Delta$.
\end{definition}

Be aware that here the notion of absolute continuity is the usual one on differential manifolds. Later in \S $6$ (\ref{L1bound}), we will define absolutely continuous curves in metric measure space.

Using the sub-Riemannian metric, one can introduce the notion of
length for horizontal curves.

\begin{definition}
 Let ${\gamma:[a,b]\rightarrow M}$ be a horizontal curve in
 ${(M,\Delta,g)}$.
The length ${L(\gamma)}$ of $\gamma$ is defined by $
{\int_a^b\sqrt{g(\dot{\gamma}(t),\dot{\gamma}(t)})dt}$.
 The sub-Riemannian or Carnot-Carath\'eodory distance ${d(x,y)}$
 between two points ${x}$ and ${y}$ on ${(M,\Delta,g)}$ is defined to be ${d(x,y)=\inf\{L(\gamma)\}}$,
 where the infimum is taken over all horizontal curves $\gamma$
 connecting ${x}$ and ${y}$. \end{definition}

The notion of distance is well-defined due to the following Theorem
\ref{thmCR}.

\begin{theorem}(Chow-Rashevskii, see also \cite{aa})\label{thmCR}
 Suppose that ${M}$ is a connected smooth manifold and $\Delta$ is a bracket-generating distribution on ${M}$. Then
\begin{enumerate}
\item for any two points on ${M}$, there is a piecewise smooth horizontal curve
connecting them;
\item the Carnot-Carath\'eodory distance d is finite and continuous on
${M}$;
\item the Carnot-Carath\'eodory distance d  induces the topology of the
manifold.
\end{enumerate}
\end{theorem}

\begin{definition}[Horizontal gradient]\label{defhg}
Let $\{X_1,\cdots,X_m\}$ be an orthogonal frame of $\Delta$ in a
sub-Riemannian manifold $(M,\Delta,g)$. For a Lipschitz function $f$ with respect to the sub-Riemannian
distance, its horizontal gradient is defined to be
\begin{displaymath}\nabla_Hf~=~X_1(f)\cdot X_1 + \cdots + X_m(f)\cdot
X_m~.
\end{displaymath}

\end{definition}

The corresponding horizontal Sobolev space $W_{H}^{1,2}(M)$ with
norm
\begin{equation}
\|f\|_H =\left(\int_M {f^2 + \langle{\nabla_Hf,\nabla_Hf}\rangle dV}\right)^\frac{1}{2}
\end{equation}
 is a Hilbert space, where $dV$ is a fixed Riemannian volume.

\begin{definition}[Radon-Nikodym property]\label{defrnp}
A Banach space V is said to have Radon-Nikodym property, if every
Lipschitz function  $f: \textbf{R}^k\rightarrow V$ is differentiable
almost everywhere.
\end{definition}

According to an early result of Gelfand, separable dual spaces and
reflexive spaces have the Radon-Nikodym property.

Finally, we recall the definition of Baire set and Baire category
theorem, which will be used in \S \ref{sect cg}.

\begin{definition}[First category set and second category set]
A set $A$ in a topological space $S$ is called nowhere dense if
 for every nonempty open set $U\subset S$, there is a
nonempty open set $V\subset U$ with $A\cap V = \emptyset$. A union
of countably many nowhere dense sets is called a set of first
category. Sets which are not of first category are said to be of
second category.
\end{definition}

\begin{definition}\label{defbp}
A subset of any topological space is said to have the property of Baire if it can be represented in the form $G\Delta A = (G-A)\bigcup(A-G)$, where $G$ is an open set and $A$ is a first category set.
\end{definition}

\begin{definition}\label{defbpf}
A map $f$ between two topological spaces is said to have the property of
Baire if $f^{-1}(V)$ has the property of Baire for every open set $V$.
\end{definition}

\begin{theorem}[Baire category theorem]\label{thmbct}
Let $(S,d)$ be any complete metric space. Suppose $N$ is a first
category set of $S$, then $S\backslash N$ is dense in $S$.
\end{theorem}

For more details about Baire category theory, one can see \cite{jo}.

\section{Embeddings of the Carnot groups}\label{sect cg}

In this section, all the Carnot groups are assumed to be {\it nonabelian} as said in Remark \ref{remarkcgnab}. Our main goal is to prove Theorem \ref{thmcn} which is one of the
main ingredients of the proofs of our main results.

\begin{theorem}\label{thmfc}
Let $g$ be a Borel map from a complete metric space $(M_1,d_1)$ to a
separable metric space $(M_2,d_2)$. Then there is a first category
set $F\subseteq M_{1}$ such that $g$ is continuous on $M_1\setminus
F$ (dense in $M_1$).
\end{theorem}

\begin{proof}
Take a countable topological basis $\{U_1,U_2,\cdots ,U_n,\cdots\}$
of $M_2$. Because $g$ is a Borel map, it has property of Baire (see Def. \ref{defbpf}). Then
\begin{eqnarray}
 g^{-1}(U_i) & = & {G_i\bigtriangleup F_i} \nonumber \\
    &=& (G_i\bigcup F_i)\setminus(G_i\bigcap F_i)\nonumber,
 \end{eqnarray}
 where $G_i$ are open sets of $M_1$ and $F_i$ are first category sets of $M_1$.
  Let \[ F =\bigcup\limits_{i=1}^{\infty}F_i,\] then $F$ is a first category set
of $M_1$. Set \[ g_1 = g\mid_{M_1\setminus F},\]  then $g_1$ is
continuous on $M_1\setminus F$. In fact,
\begin{eqnarray*}
 g_1^{-1}(U_i) &=& g^{-1}(U_i)\bigcap (M_1\setminus F)\\
  &=& (G_i\bigtriangleup F_i)\bigcap (M_1\setminus F)\\
  &=& G_i\bigcap (M_1\backslash F)
\end{eqnarray*}
i.e., $g_1^{-1}(U_i)$ are open sets of $M_1\backslash F$. By Baire
category Theorem \ref{thmbct}, $M_1\backslash F$ is dense in $M_1$.
\end{proof}

Let $G$ be a Carnot group with Lie algebra
\begin{equation}
\begin{split}
  \mathfrak{g} = \sum\limits_{k=1}^{s}V_k,\qquad  V_{k+1}=[V_1,V_k]
\end{split}
\end{equation}
where the dimension of $V_k$ = $m_k-m_{k-1}$, $2\le k \le s$ and the
dimension of $V_1$ = $m_1$.

We can construct a basis $\{X_1,X_2\cdots X_n\}$ for $\mathfrak{g}$
with respect to the above decomposition, where $n = \dim G$.
Firstly, fix a basis $\{X_1, . . . , X_{m_1}\}$ of $ V_1$, then
consider all brackets $[X_i, X_j ]$ for  $i, j = 1, \dots , m_1 $.
Since $[V_1, V_1] = V_2$, we can find among such brackets a basis
for $V_2$. Pick such a basis and denote by $\{X_{m_1+1}, \dots,
X_{m_2}\}$. Similarly, extract a basis $\{X_{m_2+1}, \dots ,
X_{m_3}\}$ for $V_{m_3}$ from the set of $[X_i, X_j ]$, for $i = 1,
\dots, m_1$ and $j = m_1 + 1, \dots, m_2$. And so on. In such a way,
we have constructed a basis $\{X_1, . . . ,X_n\}$ of $\mathfrak{g}$
such that
\begin{enumerate}
\item For any $1\le j \le s$, $\{X_{{m_{j-1}}+1}, . . . ,X_{m_j}\}$ is a basis of $V_j$.
\item For any $i = m_1 + 1,\dots, n $, there are integers $d_i , l_i , k_i$, such that
$X_i \in V_{d_i}$ , $X_{l_i} \in V_1$, $X_{k_i} \in V_{d_{i}-1}$, and
$ X_i = [X_{l_i} , X_{k_i}]$.
\end{enumerate}

We can now define a sub-Riemannian structure on $G$ with respect to
the Lie algebra decomposition given above. Let $\Delta$ be the
distribution on $G$ which is generated by the left invariant vector
fields, still denoted by $X_1,\dots ,X_{m_1}\in V_1$, where
$m_1=\dim V_1$. We define a left invariant sub-Riemannian metric on
$M$ by
\begin{equation}
 \langle X_i,X_j\rangle = \delta_{i,j},\,\,\,\,\, 0\le i,j \le m_1. \nonumber
\end{equation}
Then $(G,\Delta,{\langle \cdot,\cdot\rangle})$ is a sub-Riemannian
manifold. We denote the induced Carnot-Carath\'eodory distance
function on $G$ by $d$.

Using the basis $\{X_1,X_2\cdots X_n\}$, we can construct a privileged
coordinate system at $x\in G$ as follows: we define a map \[\Psi:
\textbf{R}^n \to G\] by \[(s_1,s_2,\cdots ,s_n) \mapsto
\exp(s_nX_n)\cdot \exp(s_{n-1}X_{n-1})\cdot \cdots \cdot
\exp(s_1X_1)\cdot x.\] Obviously, $(s_1,s_2,\cdots ,s_n)$ gives a
local coordinate system at $x$.

\begin{theorem}\label{thmpc}
Let $(s_1,s_2,\cdots,s_n)$ be  a privileged coordinate system at $x$ in a Carnot group $G$, constructed as above.
There exist constants $c$ and $e$ such that if $d(x,q)< e$, then
\[\frac{1}{c}\|s\|\le d(x,q)\le c\|s\|,\] where $s=(s_1,s_2,\cdots
,s_n)$ is the coordinate of $q$, and \[\|s\|=
|s_1|^{\frac{1}{w_1}}+|s_2|^{\frac{1}{w_2}}+\cdots
+|s_n|^{\frac{1}{w_n}}\] with $w_i = k$ if $m_{k-1}< i \le m_{k}$
and $m_0 = 0$.
\end{theorem}
The proof and the literature about the theorem can be found in \cite{aa}.

With these preparations, we now consider the Lipschitz maps between
a Carnot group and a Banach space $(V,\|.\|)$ with Radon-Nikodym
property. If there is a Lipschitz map
\begin{displaymath} f:G \rightarrow V,
\end{displaymath}  i.e.,
\begin{displaymath}
d(f(x_1),f(x_2))\le L d(x_1,x_2),~\forall x_{1},x_{2}\in G,
\end{displaymath}
then we have the following:
\begin{lemma}\label{lemmabm}
The directional derivatives $X_{i}(f)$ , $1\le i
\le m_1 $ of the Lipschitz map $f: G\rightarrow V$ exist almost everywhere, and they
are Borel maps from $G$ to $V$.
\end{lemma}

\begin{proof}
By definition, $X_i(f),\,\, 1\le i\le m_1$ exists almost everywhere.\\
Let $ E_i = \{ x\in G \mid X_i(f)\,\, \textrm{does}\,\, \textrm{not}\,\, \textrm{exist}\,\,
\textrm{at}\,\, x \}$, $1\le i\le m_1$. Then
\begin{equation}
E_i=\bigcup\limits_{k=1}^{\infty}\bigcap\limits_{m=1}^{\infty}E_{k,m},\,\,\,\,
1\le i\le m_1,
\end{equation}
where \[ E_{k,m}= \bigcup_{0<|t_1|,|t_2|<\frac{1}{m}}\left\{x\in G\mid
\|\frac{f(\exp(t_1 X_i)\cdot x)- f(x)}{t_1}-\frac{f(\exp(t_2
X_i)\cdot x)- f(x)}{t_2}\| > \frac{1}{k}\right\}.\] It is easy to see that $E_{k,m}$
are open sets, so $E_i$, $1\le i\le m_1$ is Borel set. By the
definition of Radon-Nikodym property and Fubini theorem, we have
$\mu(E_i)= 0$ for the Lebesgue measure $\mu$ on G.\\
Take a sequence $\{t_n\}$ satisfying $t_n > 0$ and
$\lim\limits_{n\to\infty}t_n = 0 $. Let
\begin{equation} F_n(x)= \frac{f(\exp(t_n X_i)\cdot x)-f(x)}{t_n}.
\end{equation}
For any $n$, $F_n$ are continuous functions on $G$, so they are
Borel maps. Because $F_n(x)\to X_i(f)(x)$ as $n \to \infty $, for $
x \in G\backslash E_{i} $, $X_i(f)$ are Borel maps from $G$ to $V$
for $1\le i\le m_1$.
\end{proof}

\begin{lemma}\label{lemmads}
There is a first category set $F$ such that $X_i(f)$ is continuous
on the dense subset $G\backslash F $ of $G$ for $1\le i
\le m_1 $.
\end{lemma}

\begin{proof}
Because $(G,d)$ is a complete separable metric space,
the Lipschitz map $f$ takes values in a closed separable subspace
$V^1$ of $V$. Hence, $X_i(f)$ ($1\le i\le m_1$) takes values in $V^1$
as well. Then the  conclusion is deduced from Theorem \ref{thmfc}.
\end{proof}

The proof of the next lemma is inspired by the method used by
Cheeger and Kleiner in their proof of Theorem $6.1$ in \cite{jc}.

\begin{lemma}\label{lemma bs}
If $x\in G\backslash F $, then for any $0 <\varepsilon < 1 $, we have
\[ \|f(x^{\prime})- f(x)\| \le hc\varepsilon\cdot d(x^{\prime},x),\]
where $c$ and $h$ are constants and $x^{\prime}= exp(t^2[X_i,X_j])\cdot
x $,  $[X_i,X_j]= X_k $ for some $m_1< k\le m_2$ and given $i,j\in [1,m_1]$.
\end{lemma}

\begin{proof}
For any $1\le i\le m_1$, $X_i(f)\mid_{G\backslash F}$ are continuous and $G\backslash F$ is dense in $G$ (Theorem    \ref{thmfc}). So there exists $\delta >
0$ such that if $x_1\in {G\backslash F} $ and $d(x,x_1)< \delta$,
then we have \[ |X_i(f)(x) - X_i(f)(x_1)| < \varepsilon,\,\,\,1\le i \le m_1.\]

Take a neighborhood $B_r(x)\subset G$ and $r< \delta $, then
$(G\backslash F)\bigcap B_r(x)$ is dense in $B_r(x)$.

Using $\eta_2, \eta_1, \eta_4, \eta_3$ to denote the four
vector fields: $X_i, X_j,-X_i,-X_j $,  where $1\le i,j\le m_1$. Then
we have the following equality:
\begin{eqnarray*}
\exp(t^2[X_i,X_j])(x)&=&\exp(-tX_i)\circ\exp(-tX_j)\circ\exp(tX_i)\circ\exp(tX_j)(x)\\
&=&\exp(t\eta_4)\circ\exp(t\eta_3)\circ\exp(t\eta_2)\circ\exp(t\eta_1)(x).
\end{eqnarray*}
Let ${\gamma}_i(t)=\exp(t\eta_i)\circ\exp(\eta_{i-1})\circ\dots\circ\exp(\eta_1)(x)$,
  then $\gamma_i$ is an integral curve
of $\eta_i$, $1\le i\le 4$.

By the denseness of $G\backslash F$, when $t< 1 $,
 there exists an integral curve $\tau_i$  of $\eta_i$ starting from $\tau_i(0)\in {(G\backslash F)\bigcap B_r(x)}$ such that
\begin{displaymath}
d(\tau_i(s),\gamma_i(s)) < \varepsilon t,\,\,\,\,s \in[0, t].
\end{displaymath}

We claim that
\[ \frac{1}{t}\int_0^t\|\eta_if(\tau_i(s))-\eta_if(x) \|ds < 2\varepsilon.\]
In fact, by the Lebesgue's lemma, \[ \lim\limits_{t\to
0}\frac{1}{t}\int_0^t\|\eta_if(\tau_i(s))-\eta_if(\tau_i(0)) \|ds =
0.\] So when $t$ is sufficiently small, we have
\begin{eqnarray*}
\frac{1}{t}\int_0^t\|\eta_if(\tau_i(s))-\eta_if(x)\|ds\
 &\le& \frac{1}{t}\int_0^t\|\eta_if(\tau_i(s))-\eta_if(\tau_i(0) \| ds+ \|\eta_if(\tau_i(0))-\eta_if(x)\|ds\\
 &=& \frac{1}{t}\int_0^t\|\eta_if(\tau_i(s))-\eta_if(\tau_i(0)) \|ds + \|\eta_if(\tau_i(0))-\eta_if(x)\|ds\\
 &<& 2\varepsilon.
 \end{eqnarray*}

Therefore
\begin{eqnarray*}
  \|f(x^{\prime})-f(x)\| &=& \|f\circ{\gamma_4}(t)-f\circ{\gamma_1}(0)\| \\
   &\le & \|f\circ{\tau_4}(t)-f\circ{\tau_1}(0)\|+ 2L\varepsilon t\\
   &\le& \|\sum_{i=1}^{4}(f\circ{\tau_i}(t)-f\circ{\tau_i(0)})\| + 8L\varepsilon t \\
&=&\|\sum_{i=1}^{4}\int_0^t\eta_i(f)(\tau_i(s))ds \| + 8L\varepsilon t \\
&\le &\|\sum_{i=1}^{4}\int_0^t\eta_i(f)(\tau_i(s))-\eta_i(f)(x)ds\|
+\|\sum_{i=1}^{4}(\eta_i(f)(x))t\| + 8L\varepsilon t \\
    &\le& \|\sum_{i=1}^{4}(\eta_i(f)(x))t\| + 8L\varepsilon t + 8t\varepsilon\\
    & = & 8tL\varepsilon + 8t\varepsilon = ht\varepsilon,
\end{eqnarray*} where $h = 8L + 8$.

For sufficiently small $t$, we have \[{\frac{t}{c}} \le d(x^{\prime},x)\le
ct,\]  where $x^{\prime}= \exp(t^2[X_i,X_j])\cdot x$ and the constant
$c$ depends  only on $x$. This is a direct consequence of
Theorem \ref{thmpc} bearing in mind the condition that $[X_i,X_j]=X_k$ for some $k\in (m_1,m_2]$.

 By the above results, we get
\begin{eqnarray}
\|f(x^{\prime})-f(x)\| &\le & ht\varepsilon\nonumber \\
&\le & hc\varepsilon\cdot d(x^{\prime},x).\nonumber
\end{eqnarray}
\end{proof}

\begin{theorem}\label{thmcn}
 Let $(G,d)$ be a Carnot group with a Carnot-Carath\'eodory
metric $d$. Then it does not admit a bi-Lipschitz embedding into any
Banach space with the Radon-Nikodym property.
\end{theorem}

\begin{proof}
It follows from the Lemmata \ref{lemmabm}, \ref{lemmads} and
\ref{lemma bs}.
\end{proof}

\section {Gromov-Hausdorff Convergence}\label{sect gh}

 The goal of this section is to complete the proof of Theorem \ref{thm 1.3}.
 We first  prove the following Theorem \ref{thmcm} which is a refined version of a similar result which we learned from Rong \cite{R}.
By direct application of Theorem \ref{thmcm} to the case that $(X,p)$ and $(Y,q)$ are metric tangent cones, we get corollary \ref{corotc} which is important
in the proof of Theorem \ref{thm 1.3}.

\begin{theorem}[Convergence of $L$-biLipschitz maps]\label{thmcm}
Suppose $\lim\limits_{i\to\infty}(X_i,p_i)=(X,p)$ and
$\lim\limits_{i\to\infty}(Y_i,q_i)=(Y,q)$ are two sequences of
locally compact length spaces which converge in the Gromov-Hausdorff
sense to two locally compact length spaces respectively. If
$\{f_i:(X_i,p_i)\to (Y_i,q_i)\}_{i=1}^{\infty}$ is a sequence of
$L$-biLipschitz maps and $f_i(p_i)= q_i $ for all $i$ with $L$ a
fixed constant. Then there is a subsequence of $f_i$ converging to an
$L$-biLipschitz map $f:(X,p)\to(Y,q)$ and $f(p)=q$ in the sense that
if $\lim\limits_{i\to \infty}{x_i = x}$, then $\lim\limits_{i\to
\infty}{f_i(x_i)= f(x)}$.
\end{theorem}

\begin{proof}
Firstly, we prove the theorem in the case that $X$ and $Y$ are
compact metric spaces. In this case, we can omit the base points.
Let $A=\{a_1,a_2,a_3,a_4,\cdots\}$ denote a countable dense subset
of $X$. We first define an $L$-biLipschitz map from $A$ to $Y$:
\begin{displaymath}
f: A \to Y.
\end{displaymath}
This extends uniquely to a map from $X$ to $Y$. For $a_1$, let
$x_i\in X_i$ such that $\lim\limits_{i\to \infty}{x_i = a_1}$, in
$\coprod\limits_{i=1}^{\infty} X_i\coprod X$. By the equicontinuity,
there exist a subsequence $\{X_{i_1}\}$ of $\{X_i\}$ and $x_{i_1}\in
X_{i_1}$ such that $\lim\limits_{i_1\to \infty}{x_{i_1}= a_1}$ and
$\lim\limits_{i_1\to \infty}{f_{i_1}(x_{i_1})= b_1}$. We define
$f(a_1)= b_1$. Repeating this process, and by the standard diagonal
argument, we can find a subsequence $\{X^{\prime}_j\}$ of $\{X_i\}$
such that$\lim\limits_{j\to\infty}X^{\prime}_j=X$  and if
$\lim\limits_{l\to\infty}x_{k,l}=a_k$, then
$\lim\limits_{l\to\infty}f_l(x_{k,l})= b_k$, and we define
$f(a_k)=b_k$, where $k\in\mathbb{N}$. Obviously, the map $f: A \to
Y$ is $L$-biLipschitz and can be extended uniquely to an
$L$-biLipschitz map $f: X\to Y$.

 We now consider the case when either $X$ or $Y$ is not
  compact. By the definition of Gromov-Hausdorff convergence of metric spaces, for any $r>0$, we can assume that the
  closed balls $\lim\limits_{i\to\infty}{B_r(p_i)}= B_r(p)$ and $\lim\limits_{i\to\infty}{B_r(q_i)}= B_r(q)$.
  Applying the above result, we conclude that there is an
  $L$-biLipschitz map $f_r: B_r(p)\to B_r(q)$ such that $f_r(p)= q
  $. Note that $f_r$  depends only on $A_r = B_r(p)\bigcap A$.
  Let $ X=\bigcup\limits_{s=1}^{\infty}B_s(p)$. Clearly, for $s,t\in\mathbb{N}, s<t$, $B_s(p)\subseteqq
  B_t(p)$ and $f_s=f_t|B_s(p) $. We get the desired $L$-biLipschitz map $f:(X,p)\to(Y,q)$.
\end{proof}

\begin{corollary}\label{corotc}
Suppose $f:(X,d_X)\to(Y,d_Y)$ is an L-biLipschitz map between two
metric spaces $(X,d_X)$ and $(Y,d_Y)$.  If $(T_xX,d_1)$ is a metric
tangent cone of $(X,d_X)$ at $x$ and $(T_{f(x)}Y,d_2)$ is the unique
metric tangent cone of $(Y,d_Y)$ at $f(x)$, then there exists an
L-biLipschitz map:
\begin{displaymath} D_{x}f: (T_xX,d_1) \to (T_{f(x)}Y,d_2).
\end{displaymath}
 \end{corollary}

\begin{theorem}(Mitchell, \cite{jm}; Bellalche, \cite{ab})\label{thmtc}
 For a sub-Riemannian manifold $(M,\Delta,g)$, the tangent cone of $(M,d)$
at every regular point is  $(G,d^c)$ where G is a Carnot Lie group
with left invariant Carnot-Carath\'eodory metric $d^c$.
\end{theorem}

\begin{remark}\label{remarktc}
If the generic nonholonomy degree of $(M,\Delta,g)$ is not smaller
than $2$, then the set of the regular points of $M$ is open and dense in $M$ as pointed out by Bellaihe (\cite{ab}, the last paragraph of p.31). By the Theorem \ref{thmtc}, the tangent cone $(G,d^c)$ of $(M,d)$ at a regular point is a nonabelian Carnot Lie group with left invariant Carnot-Carath\'eodory metric $d^c$.
\end{remark}

\begin{proof}[Proof of Theorem \ref{thm 1.3}]
Let $(M,d)$ be a sub-Riemannian manifold with generic nonholonomy
degree $\ge 2$ and sub-Riemannian distance $d$. Suppose $(V,\|\cdot\|)$ is a
Banach space satisfying the Radon-Nikodym property with the distance
function $d^*$ induced by the norm $\|\cdot\|$.

We proceed by contradiction, suppose there is a bi-Lipschitz map
\begin{displaymath}
f: (M,d)\to(V,d^*).
\end{displaymath}
By Remark \ref{remarkrpd}, the set of the regular points of $M$ with the
degree of non-holonomy satisfying $\ge 2$ is dense in
$M$. So we can get a bi-Lipschitz map $D_{x}f$ from the tangent cone
$(T_{x}M,d_1)$ at a regular point $x$ of $M$ to the tangent cone $(T_{f(x)}V,d_2)$
of $V$, i.e.,
\begin{displaymath}
D_{x}f: (T_{x}M,d_1)\to (T_{f(x)}V,d_2).
\end{displaymath}
Because $(V,d^{*})$ is a Banach space, the tangent cone of $V$ at $f(x)$ is
\begin{displaymath}
(T_{f(x)}V,d_2)= (V,d^*).
\end{displaymath}

By Remark \ref{remarktc}, $(T_xM,d_1)$ is a non-commutative Carnot
group with left invariant Carnot-Carath\'eodory metric $d_1$.
However, it is impossible by Theorem \ref{thmcn}. Therefore, every
sub-Riemannian manifold whose generic degree of non-holonomy is not
smaller than $2$ can not be bi-Lipschitzly embedded in any Banach
space with Radon-Nikodym property.
\end{proof}

\begin{proof}[Proof of Corollary \ref{coroesrnercd}]
We still go on by contradiction. If there is a biLipschitz embedding of a complete sub-Riemmannian manifold into a metric measure space $(M,d,m)$ satisfying $RCD(K,N)$ (Def. \ref{defrcd}), where $(M,d)$ is a length space and $supp(m)= M$. As in the above argument, through the blowup analysis (Theorem \ref{thmcm}), we get a biLipschitz embedding from a Carnot group with Carnot-Carath\'eodory metric to a Euclidean space (Mondino-Naber, \cite{am}, see also Theorem \ref{thmrcd}). However, it is impossible by Theorem \ref{thm 1.3}.
\end{proof}

\begin{remark}
In the above proof of Corollary \ref{coroesrnercd}, the condition that the metric measure space $(M,d,m)$ satisfies $supp(m)= M$, is essential. Because, by the Corollary 2.4 in \cite{kts}($\uppercase\expandafter{\romannumeral2}$), under the condition, the $CD(K,N)$ space $(M,d,m)$ is locally compact.
\end{remark}

\section{Curvature-dimension condition}\label{sect cd}

In this section, we recall the definition of curvature-dimension
 condition and some of its basic properties.

Given a metric measure space $(M,d,m)$ and a number
 $N \in \mathbf{R}$ with $N \geq 1$, we define the Renyi entropy functional $\mathcal{S}_{N}(\cdot|m):\mathcal{P}_{2}(M,d)\rightarrow\mathbf{R}$ with respect to $m$ by
 $$\mathcal{S}_{N}(\nu|m)= -\int_{M}{\rho^{\frac{-1}{N}}d\nu}$$
where $\mathcal{P}_{2}(M,d)$ denotes the $L_{2}$-Wasserstein space
of probability measures on $M$ and $\rho$ is the density of the
absolutely continuous part $\nu^{c}$ in the Lebesgue decomposition
$\nu=\nu^{c}+\nu^{s}=\varrho m+\nu^{s}$ of
$\nu\in\mathcal{P}_{2}(M,d)$.

\begin{definition}\label{defcd}
Given two numbers $K,N\in\mathbf{R}$ with $N\geq1$, we say that a
metric measure space $(M,d,m)$ satisfies the curvature-dimension
condition $CD(K,N)$ if for each pair
$\nu_{0},\nu_{1}\in\mathcal{P}_{2}(M,d)$, there is an optimal
coupling $q$ of $\nu_{0}^c=\varrho_{0}m$ and $\nu_{1}^c=\varrho_{1}m$,
and a geodesic $\Gamma:[0,1]\rightarrow\mathcal{P}_{2}(M,d)$
connecting $\nu_{0}$ and $\nu_{1}$, with
\begin{equation*}
\mathcal{S}_{N^{'}}(\Gamma(t)|m)\leq-\int_{M\times
M}[\tau^{(1-t)}_{K,N^{'}}(d(x_{0},x_{1}))\varrho_{0}^{\frac{-1}{N^{'}}}(x_{0})
+\tau^{(t)}_{K,N^{'}}(d(x_{0},x_{1}))\varrho_{1}^{\frac{-1}{N^{'}}}(x_{1})
]dq(x_{0},x_{1})
\end{equation*}
for all $t\in[0,1]$ and all $N^{'}\geq N$,
where

\begin{equation*}\tau^{(t)}_{K,N}(\theta)=
\begin{cases} \infty & \mbox{if $K \theta \geq {(N-1)\pi^{2}}$}\\ t^{\frac{1}{N}}(\frac{\sin(t\theta\sqrt{K(N-1)})}{\sin(\theta\sqrt{K(N-1)})})^{1-\frac{1}{N}}
 & \mbox{if $0<K\theta<(N-1)\pi^{2}$}\\
t & \mbox{if $K\theta^{2}=0$ or if $K\theta^{2}<0$, $N=1$}\\
{t^{\frac{1}{N}}}{(\frac{\sinh(t\theta\sqrt{-K(N-1)})}{\sinh(\theta\sqrt{-K(N-1)})})^{1-\frac{1}{N}}}&\mbox{if$K\theta^{2}<0$,
$N>1$}.
\end{cases}
\end{equation*}

\end{definition}

\begin{remark}
A coupling $q$ of two probability measures $\mu$ and $\nu$ on a
metric space $M$ is a probability measure on the product space
$M\times M$ whose marginals are the given measures $\mu$ and $\nu$.
For the condition of an coupling being optimal, please refer to
\cite{kts} and \cite{jl}.

\end{remark}

 We list some properties of curvature-dimension condition to be used later in this paper, and refer to \cite{kts} for more complete discussions.

 \begin{lemma}\label{lemmms}\label{lemma cdp}
Let $(M,d,m)$ be a metric measure space which satisfies the
$CD(K,N)$ condition for some pair of real numbers $K$ and $N$. Then
the following properties hold.
\begin{enumerate}
\item Each metric measure space $(M^{'},d^{'},m^{'})$ which is
isomorphic to $(M,d,m)$ satisfies the $CD(K,N)$ condition.
\item For each $\alpha,\beta >0$, the metric measure space $(M,\alpha d,\beta
m)$ satisfies the $CD(\alpha^{-2}K,N)$ condition.
\item For each convex subset $M^{'}$ of $M$, the metric measure space
$(M^{'},d,m)$ satisfies the same $CD(K,N)$ condition.
\end{enumerate}

\end{lemma}

\section{Cheeger energy and Sobolev space on metric measure spaces}\label{sect cs}

In this section, we will give the proof of Theorem \ref{thm1.1}, Theorem \ref{thm1.2} and Theorem \ref{thmih}. We firstly recall some basics about the differential structures of metric measure spaces, which will be used in this section.

In order to define the Sobolev space $W^{1,2}(M,d,m)$ on metric measure space, we need to introduce the weak gradients. As far as we know, there are four notions of weak gradients. However, under our assumptions, they are all equivalent (\cite{ew}, Theorem 7.4 and Further comments and extensions). In the following, we only need two of them, namely the Cheeger's gradient $|\nabla{f}|_{C,2}$ and $2$-weak upper gradient $|Df|_{w,2}$.

Firstly, recall the notion of absolute continuity of curves on the metric measure spaces which plays an important role in the definition of weak gradients.

A curve $\gamma:[0,1]\rightarrow M$ is said to be absolutely continuous if
\begin{equation}\label{L1bound}
\begin{split}
 d(\gamma(t),\gamma(s))\leq\int_{s}^{t}g(r)dr \qquad\forall s,t\in[0,1],s\leq t
\end{split}
\end{equation}
for some $g\in L^{1}(0,1)$. If $\gamma$ is absolutely continuous,
the metric derivative $|\dot{\gamma}|:[0,1]\rightarrow[0,\infty]$ is
defined by
\begin{equation*}
|\dot{\gamma}|:=\underset{h\rightarrow0}{\lim}\frac{d(\gamma(t+h),\gamma(t))}{|h|}.
\end{equation*}
One can  show that the limit exists for almost every
$t$, and $|\dot{\gamma}| \in L^{1}(0,1)$, and it is the minimal
$L^{1}$-function (up to Lebesgue negligible sets) for which the
bound (\ref{L1bound}) holds.
\begin{remark}\label{ccmd}
For the case of regular sub-Riemannian manifold with Carnot-Carath\'eodory distance, the absolutely continuous curves $\gamma$ are almost differentiable and Horizontal. The metric derivatives $|\dot{\gamma}|$  of $\gamma$ at the differentiable points are the usual norm of the tangent vectors of the curves (\cite{mksv}, Corollary 2.8.6, Property 3.2.2 and Theorem 3.2.10).
\end{remark}

Following \cite{ew}, we give the following definitions.

\begin{definition}(Upper gradient)
Suppose $f$ is a Borel function on $X$, we say that a Borel function $h$ is an upper gradient of $f$, if the inequality

\begin{equation}
     |f(0)-f(1) | \leq \int_{\gamma}{h(\gamma)}|\dot{\gamma}|ds
\end{equation}
holds for all absolutely continuous curves $\gamma:[0,1]\rightarrow X$.
\end{definition}

The next notion is due to Cheeger, which is defined via upper gradient by a relaxation procedure.

\begin{definition}($2$-relaxed upper gradient)\label{defrug}
We say that $h$ $\in$ $L^{2}(M,m)$ is a $2$-relaxed upper gradient of $f$ $\in$ $L^{2}(M,m)$ if there exist $\tilde{h}$ $\in$ $L^{2}(M,m)$, functions $f_{n}\in
L^{2}(M,m)$ and upper gradient $h_{n}$ of $f_{n}$ such that:
\begin{enumerate}
\item $f_{n}\rightarrow f$ in $L^{2}(M,m)$ and $h_{n}$ weakly converge to $\tilde{h}$ in $L^{2}(M,m)$;
\item $\tilde{h}\leq h$.
\end{enumerate}
We say that $h$ is a minimal $2$-relaxed upper gradient of $f$ if its $L^{2}(M,m)$ norm is minimal among $2$-relaxed upper gradients. We denote by
$|\nabla{f}|_{C,2}$ the minimal $2$-relaxed upper gradient.

\end{definition}

We will denote by $AC^{2}([0,1],M)$ the class of absolutely
continuous curves with metric derivative in $L^{2}(0,1)$.

In order to introduce the concept of weak upper gradient, some special probability measure is
defined on the space of continuous paths of $M$.

\begin{definition}
Let $(M,d,m)$ be a metric measure space and $\pi$ be a probability of $C([0,1],M)$. We say that $\pi$ has bounded compression provided there exists $N>0$
such that
\begin{gather}
(e_t)_{\sharp}\pi \leq N\cdot m,\quad\forall t \in [0,1],
\end{gather}
 where $e_{t}$ is the evaluation function from $C([0,1],M)$ to $M$, at $t$.\\
 We say that $\pi$ is a test plan if it has bounded compression and is concentrated on $AC^{2}([0,1],M)$ such that
\begin{gather}
 \int\int_{0}^{1}|\dot{\gamma}|^{2}dtd\pi(\gamma)< \infty.
\end{gather}
\end{definition}

\begin{definition}\label{defwg}
Let $f:M\rightarrow \overline{\mathbf{R}}$ be any Borel function. A
function $G$ $\in$ $L^{2}(M,m)$ is called a weak upper
gradient of $f$ if
\begin{equation}
|f(\gamma(0))-f(\gamma(1))| \leq
\int_0^1{G(\gamma(t))|\dot{\gamma}(t)|}dt < \infty
\end{equation}
for almost every curve $\gamma:[0,1]\rightarrow M$ in
$AC^{2}([0,1];M)$, under every test plan $\pi$.

\end{definition}

It turns out (\cite{la}, P.321) that there is a weak upper
gradient $|D(f)|_{w}:M\rightarrow [0,\infty]$ having the property
that
$$|D(f)|_{w}(x)\leq G(x)$$
for $m$-a.e  $x\in M$, where $G$ is any weak upper gradient of $f$. The
 function $|D(f)|_{w}$ will be called the minimal weak upper gradient
 of $f$.

In terms of minimal weak upper gradient, one defines the Cheeger energy of $f:M\rightarrow \overline{\mathbf{R}}$ by
\begin{equation}
Ch(f):=\frac{1}{2}\int_{M}|Df|_{w}^{2}dm
\end{equation}

 The Sobolev space $W^{1,2}(M,d,m)$ is by definition the space of $L^{2}(M,m)$-
functions having finite Cheeger energy, and it is endowed with the
natural norm $\|f\|_{W^{1,2}}^{2}:= \|f\|_{L^{2}}^{2} + 2Ch(f)$
which makes it into a Banach space.

\begin{remark}\label{}
In general, $W^{1,2}(M,d,m)$ is not a Hilbert space. For instance, on
a smooth Finsler manifold the space $W^{1,2}$ is Hilbert if and only
if it is actually Riemannian manifold.
\end{remark}

\begin{definition}[Infinitesimally Hilbertian space]\label{defihs}
 A metric measure space $(M,d,m)$ is said to be
an infinitesimally Hilbertian space if $(M,d)$ is a complete separable
metric space, $m$ is a locally finite non-negative measure on $M$, and the
Sobolev space $W^{1,2}(M,d,m)$ is a Hilbert space.
\end{definition}

\begin{definition}\label{defrcd}
If an infinitesimally Hilbertian space $(M,d,m)$ satisfies
curvature-dimension condition $CD(K,N)$, we say $(M,d,m)$ is a
$RCD(K,N)$ space.
\end{definition}

In \cite{am}, Mondino and Naber proved that
\begin{theorem} (Mondino-Naber, \cite{am})\label{thmrcd}
Suppose $(X, d, m)$ is a complete and separable metric space with $m$ a locally finite nonnegative complete Borel measure satisfying $RCD(K,N)$ with $N>1$, then for
m-a.e. point the tangent cone is an Euclidean space of dimension at
most $N$.
\end{theorem}

\begin{remark}\label{remarkcd}
In fact, reduced curvature-dimension condition $CD(K,N)^{*}$ is
used in \cite{am}, which is slightly weaker than the curvature-dimension
condition $CD(K,N)$.
\end{remark}

We will prove that the minimal weak upper gradient and the horizontal gradient (Def. \ref{defhg}) coincide on the Carnot-Carath\'{e}odory spaces. The following lemma tells us that if the function is continuous, then the two concepts are the same up to a measure zero set. However it is still not enough, and we need a stronger result to prove the coincidence.

\begin{lemma}\label{lemma hg}
Let $(M,\Delta,g)$ be a regular sub-Riemannian manifold with a Riemmanian volume $m$. Suppose $f\in L^{1}_{loc}(M,m)$ is a nonnegative function. If $f$ is an upper gradient of a function $u$ which is continuous. Then the distributional derivatives $\nabla_{H}(u)$ is locally integrable and $|\nabla_{H}(u)|$ $\leq$ $f$  $m-a.e$ in $M$.
\end{lemma}

\begin{proof}
Obviously, we only need to check the conclusion locally. Without loss of generality, we may assume that $M$ is an open domain of Euclidean space, and the horizontal distribution $\Delta$ is generated by global orthogonal vector fields $\{X_{1},\cdots,X_{r}\}$, where $r$ is the rank of $\Delta$. The measure $m$ can be written in the form $\rho dV$, where $\rho$ is a smooth function and $dV$ is the Lebesgue measure. By the definition of distributional derivatives, we have:
\begin{align*}
 \underset{M}{\int} X_{i}(u)\cdot h\cdot\rho dV = \underset{M}{\int} u\cdot X_{i}^{\ast}(h)\cdot\rho dV,
\end{align*}
where $h$ is any smooth function with compact support on $M$, and $X_{i}^{\ast}$ is the adjoint operator of $X_{i}$ with respect to $\rho dV$.
This formula can be rewritten as:
\begin{align*}
\underset{M}{\int}(\rho\cdot X_{i})(u)\cdot h\cdot dV = \underset{M}{\int}u\cdot (\rho\cdot X_{i}^\ast)(h)\cdot dV
\end{align*}
So, $\rho\cdot X_{i}^{\ast}$ is the adjoint operator of $\rho\cdot X_{i}$ with respect to $dV$.

Let $Y_{i}$ = $\rho \cdot X_{i}$, then we get a new basis of the horizontal distribution $\Delta$. Define a new sub-Riemannian metric $g^{'}$ on $\Delta$, such that\{ $Y_{1}, \cdots, Y_{r}$\} is orthogonal. On this new sub-Riemannian manifold $(M,\Delta,g^{'})$, $\rho\cdot f$ is an upper gradient of $u$, and the  distributional horizontal gradient is $$\nabla_{H}^{'}(u) = \rho\cdot\nabla_{H}(u).$$ Hence, we reduce the problem to the case : $M$ is a domain of Euclidean space with Carnot-Carath\'{e}odory metric and the Lebesgue measure. This is exactly the Theorem 11.7 in \cite{ph}.

\end{proof}

In \cite{ew}, Ambrosio, Gigli and Savare proved the following density in energy of Lipschitz functions.

\begin{proposition}(Ambrosio-Gigli-Savare, \cite{ew})\label{propld}
If $f\in\ W^{1,2}(M,d,m)$, then there exist Lipschitz functions $f_{n}$ satisfying\\
\begin{equation}
 \underset{n\rightarrow\infty}{\lim}\left [{\int_{M}|f_{n}-f|^{2}dm + \int_{M}\big(|\nabla{f_{n}}|-|Df|_{w}\big)^{2}dm} \right ]= 0,
\end{equation}
where $|\nabla{f_{n}}|$ is defined by

\begin{equation}
|\nabla{f_{n}}|(x) = \underset{y\rightarrow x}{\overline{\lim}}\frac{|f(x)-f(y)|}{d(x,y)}.
\end{equation}

\end{proposition}

\begin{remark}
In general, the Lipschitz functions are not dense in $W^{1,2}(M,d,m)$ w.r.t $W^{1,2}$ norm, as pointed out by Gigli (\cite{ng}). So we cannot apply Lemma \ref{lemma hg} directly to prove Theorem \ref{thmih}.

\end{remark}

A further lemma known to the experts is needed, and we give a proof here for the completeness.

\begin{lemma}\label{lemmabc}
Let $(M,g)$ be a Riemannian manifold with Riemannian measure $m$, which is induced from the Riemannian volume. Suppose
$f$ is an $m$-measurable function, then there exists a Borel function $\tilde{f}$ such that
\begin{equation}
f\leq\tilde{f} \qquad  and \qquad f = \tilde{f} \quad m-a.e..
\end{equation}
\end{lemma}
\begin{proof}
Because $f$ is measurable, for any $r\in\mathbf{Q}$, the set
\begin{equation}
E_{r}=\{x\in M: f(x)\geq r\}
\end{equation}
is measurable. By the regularity of the measure $m$, we have a Borel set $A_{r}\supseteq E_{r}$ such that $m(A_{r}\setminus E_{r})$ $=$ $0$. We define
\begin{equation}
\tilde{f}(x):=\sup\{r\in\mathbf{R}:x \in \underset{q\in\mathbf{Q}:q<r}{\bigcap} A_{q}\}.
\end{equation}
By the definition of $E_{r}$, we have that
\begin{equation}
f(x)=\sup\{r\in\mathbf{R}:x \in \underset{q\in\mathbf{Q}:q<r}{\bigcap} E_{q}\}.
\end{equation}

Since $E_{q}\subset A_{q}$, for all $q\in\mathbf{Q}$, we see that $f(x)\leq \tilde{f}(x)$.
If there is an $x\in M$, such that $f(x) < \tilde{f}(x)$, then there is a $e\in \mathbf{R}$ satisfying $f(x)< e <\tilde{f}(x)$.
Then we have

\begin{equation}
x\in (\underset{q\in\mathbf{Q}:q<e}{\bigcap}A_{q}) \subset \underset{q\in\mathbf{Q}:q<e}{\bigcup}(A_q\setminus E_q).
\end{equation}
Since $m(A_q\setminus E_q)$ = $0$ for all $q\in \mathbf{Q}$, $f(x)=\tilde{f}(x)$ $m-a.e$.

By the definition of $\tilde {f}(x)$, if $x\in\underset{q\in\mathbf{Q}:q<r}{\bigcap} A_{q}$, then $\tilde{f}(x)\geq r$. On the other hand, if $\tilde{f}(x)\geq r$ then $x\in \underset{q\in\mathbf{Q}:q<r}{\bigcap} A_{q}$. So
\begin{equation}
\underset{q\in\mathbf{Q}:q<r}{\bigcap} A_{q} = \tilde{f}^{-1}([r,+\infty])
\end{equation}
i.e., $\tilde{f}$ is a Borel function.

\end{proof}

\begin{theorem}\label{thmih}
Let $(M,\Delta,g)$ be a regular sub-Riemannian manifold with measure
$m$ induced by Riemannian volume and Carnot-Carath\'{e}odory distance $d$. Then the two spaces $W^{1,2}(M,d,m)$ and $W_{H}^{1,2}(M,m)$ (Def. \ref{defhg}) coincide. So $(M,d,m)$ is an infinitesimally Hilbert space.
\end{theorem}

\begin{proof}
Suppose $f$ $\in$ $W^{1,2}(M,d,m)$, then there is a sequence of Lipschitz functions $\{f_{n}\}$ converging to $f$ in $L^{2}(M,m)$ by Proposition \ref{propld}.
Because $f_n$ is Lipschitz function for any $n$, $\nabla_H f$ exists and $|\nabla f_{n}|$ is an upper gradient of $f_{n}$. By Lemma \ref{lemma hg}, \begin{equation}
    |\nabla_{H}{f_{n}}|  \leq  |\nabla f_{n}|.
\end{equation}
Hence,
\begin{equation}
\underset{M}{\int}|\nabla_H{f_n}|^2dm \leq \underset{M}{\int}|\nabla f_n|^2dm.
\end{equation}

And also,
\begin{equation}
\begin{split}
\underset{n\rightarrow\infty}{\lim}{\int_{M}\big||\nabla{f_{n}}|-|Df|_{w}\big|^{2}dm} = 0, \\
\quad \underset{M}{\int}|Df|_w^2dm < \infty,
\end{split}
\end{equation}
then
\begin{equation}
  \underset{M}{\int}|\nabla_H{f_n}|^2dm < \infty.
\end{equation}

So $\{\nabla_{H}{f_{n}}\}$ is a bounded set in the space $L^{2}(M,\Delta)$ of square integrable sections of $\Delta$. On the other hand, $L^{2}(M,\Delta)$ is a Hibert space, so $L^{2}(M,\Delta)$ is reflexive. Hence, there is a subsequence of $\{\nabla_{H}f_{n}\}$ converging to a horizontal section $\alpha$, and $\nabla_{H}f$ $=$ $\alpha$. We can assume that
\begin{equation}
\underset{n\rightarrow\infty}{\lim}\nabla_{H}f_{n} = \alpha,
\end{equation}
\begin{equation}
|\nabla_{H}f_{n}|\rightarrow |\alpha| \qquad   in\quad L^{2}(M,m).
\end{equation}

By Lemma \ref{lemmabc}, for each $n$, we can find a Borel function $\tilde{f_{n}}$ such that

\begin{equation}
 |\nabla_{H}f_{n}| \leq \tilde{f_{n}} \qquad and \qquad |\nabla_{H}f_{n}| = \tilde{f_{n}}\qquad m-a.e..
\end{equation}

For any $n$, the $\tilde f_n$ is an upper gradient of $f_n$. By Remark \ref{ccmd}, any absolutely continuous path $\gamma:[0,1]\rightarrow M$ is almost differentiable and horizontal. Then,
\begin{equation}
\begin{split}
|f(\gamma(0))-f(\gamma(1))| \leq \int_0^1|\nabla_H f_n(\gamma(s))|\cdot|\dot{\gamma}(s)|ds\\
\leq \int_0^1|\tilde{f_n}(\gamma(s))|\cdot|\dot{\gamma}(s)|ds.
\end{split}
\end{equation}

Therefore, we get a sequence of Lipschitz functions $\{f_n\}$ and a sequence of Borel functions $\{\tilde f_n\}$, satisfying

\begin{enumerate}
\item $f_{n}\rightarrow f$ in $L^{2}(M,m)$ and $\tilde f_n \rightarrow |\alpha|$ in $L^{2}(M,m)$;
\item $\tilde{f_n}$ is a upper gradient of $f_n$, for each $n$.
\end{enumerate}

By Definition \ref{defrug}, we have $\nabla_H f = |\alpha|$ is a 2-relax upper gradient of $f$. And also, the $|\nabla f|_{C,2}$ is minimal. So,

\begin{equation}
\underset{M}{\int}|\nabla{f}|_{C,2}^2dm \leq \underset{M}{\int}|\nabla_H{f}|^2dm.
\end{equation}

By the formula (5.22) and (5.23),
\begin{equation}
\underset{M}{\int}|\nabla_{H}f|^2dm \leq \underset{M}{\int}|Df|_{w}^2dm.
\end{equation}


By \cite{ew} (Theorem 7.4 and the section "Further comments and extensions"), we get $|\nabla{f}|_{C,2} = |Df|_{w}$ $m-a.e.$ in $M$. So
 \begin{equation}
 \underset{M}{\int}|\nabla{f}|_{C,2}^2dm = \underset{M}{\int}|\nabla_H{f}|^2dm.
\end{equation}

 And also, the minimal relaxed upper gradient is unique up to a set of measure zero (\cite{cc}, Theorem 2.10), then
 \begin{equation}
 |\nabla_{H}f| = |Df|_{w} \qquad m-a.e..
\end{equation}

Therefore, we get an embedding of Banach space $W^{1,2}(M,d,m)$ into Banach space $W^{1,2}_H(M,m)$, which preserves the norms. But the space $\mathcal{C}_0^{\infty}(M)$ of smooth functions with compact support is dense in $W^{1,2}_H(M,m)$ and $W^{1,2}(M,d,m)$ contains $\mathcal{C}_0^{\infty}(M)$. So $W^{1,2}(M,d,m)$ = $W^{1,2}_H(M,m)$.
On the other hand, the Sobolev space $W_{H}^{1,2}(M,m)$ with norm
\begin{equation}
\|f\|_H =\left(\int_M ({f^2 + \langle{\nabla_Hf,\nabla_Hf}\rangle) dm}\right)^\frac{1}{2}
\end{equation}
 is a Hilbert space. Therefore, $(M,d,m)$ is an infinitesimally Hilbertian space.
\end{proof}

\begin{proof}[Proof of Theorem \ref{thm1.2}]
Suppose $(M,d,m)$ satisfies curvature-dimension condition $CD(K,N)$, where $K$, $N$ $\in$ $\mathbf{R}$ and $N>1$. By Theorem \ref{thmih}, the metric measure space $(M,d,m)$ is a $RCD(K,N)$ space.
It follows from Theorem \ref{thmrcd} that for $m-a.e.$ point the tangent cone is a Euclidean space. By Theorem \ref{thmtc}, we get a bi-Lipschitz map between Carnot group and Euclidean space. This contradicts to Theorem \ref{thm 1.3}.
\end{proof}

\begin{proof}[Proof of Theorem \ref{thm1.1}]
\hspace*{1em}Suppose a complete separable metric measure space $(M,d,m)$ satisfies $RCD(0,N)$ for $N > 1$. Fixed a point $x\in M$, consider the family of pointed  metric measure spaces:
\begin{equation}
\{(M,\lambda\cdot d, m, x):\quad \lambda\in\mathbf{R}_{>0}\}
\end{equation}
Then, by Lemma \ref{lemma cdp}, the above family of pointed metric measure spaces also satisfies $RCD(0,N)$. And also, by the Theorem 2.3 and Corollary 2.4 in \cite{kts}($\uppercase\expandafter{\romannumeral2}$), the family is uniformly doubling. So, by Lemma 3.32 in \cite{gms}, there is a sequence ${\lambda_{n}}$ of positive real numbers and $\underset{n\rightarrow\infty}{\lim}\lambda_n = 0$ such that the $(CM,d_{0},m)$ :=
$\underset{n\rightarrow\infty}\lim(M,\lambda_n\cdot d,m,x)$ exists. Then, by the stability of lower Ricci curvature bounds under pointed measured Gromov-Hausdorff convergence \cite{gms}, $(CM,d_{0},m)$ satisfies $RCD(0,N)$.\\
   \hspace*{1em} $(N,g)$ is a non-abelian nilpotent Lie group with a left invariant Riemannian metric. Pansu (\cite{pp}) proved that the
asymptotic cone of $(N,g)$ is unique and isometric to Carnot group
with a left invariant Carnot-Carath\'eodory metric.  By Corollary \ref{coroesrnercd}, we complete the proof.
\end{proof}

\section*{Acknowledgement}
{We would like to thank Professor Xiaochun Rong for his valuable
suggestions and comments. In addition, the first author is grateful to Professor
Fuquan Fang for introducing him into the subject of Sub-Riemannian geometry. We especially thank the referees for giving us very valuable advices which help us a lot to improve the presentations.


\end{document}